\documentclass[10pt]{article}
\usepackage{amsthm, amstext}
\usepackage{amsmath}
\usepackage{amsmath, amsthm, amssymb}

\usepackage{enumitem}
\usepackage{graphicx}
\usepackage{amsfonts}
\usepackage{latexsym, xypic}
\usepackage[all]{xy}
\xyoption{all}
\usepackage[colorlinks,linkcolor=blue,citecolor=blue]{hyperref}
\theoremstyle{plain}
\newtheorem{theorem}{Theorem}[section]
\newtheorem{lemma}[theorem]{Lemma}
\newtheorem{proposition}[theorem]{Proposition}

\theoremstyle{definition}
\newtheorem{definition}[theorem]{Definition}
\newtheorem{corollary}[theorem]{Corollary}
\newtheorem{example}{\sc Example}
\theoremstyle{remark}
\newtheorem{remark}{\sc Remark}
\theoremstyle{case}

\def\pi{positive implicative }

\date{}
\begin{document}

\title{\bf {Conjugate $L$-Subgroups of an $L$-group and their Applications to Normality and Normalizer}}
\author{\textbf{Iffat Jahan$^1$ and Ananya Manas $^2$} \\\\ 
	$^{1}$
	Department of Mathematics, Ramjas College\\
	University of Delhi, Delhi, India \\
	ij.umar@yahoo.com \\\\
	$^{2}$Department of Mathematics, \\
	University of Delhi, Delhi, India \\
	anayamanas@gmail.com 
	\\    }
\date{}
\maketitle

\begin{abstract}
	In this paper, the notion of the conjugate of an $L$-subgroup by an $L$-point has been introduced. Then, several properties of conjugate $L$-subgroups have been studied analogous to their group-theoretic counterparts. Also, the notion of conjugacy has been investigated in the context of normality of $L$-subgroups. Furthermore, some important relationships between conjugate $L$-subgroups and normalizer have also been established. Finally, the normalizer of an $L$-subgroup of an $L$-group has been defined by using the notion of  conjugate $L$-subgroups. 
	
\end{abstract}

\section{Introduction}

\noindent The studies of fuzzy algebraic structures began in 1971 when Rosenfeld \cite{rosenfeld_fuzzy} applied the notion of fuzzy sets to subgroupoids and subgroups. In 1981, Liu \cite{liu_op}, replaced the closed unit interval with  of 'lattices' in the definition of fuzzy sets and introduced the notion of lattice valued fuzzy subgroups. Subsequently, a number of researchers investigated fuzzy algebraic structures and generalized various concepts to the fuzzy setting. We mention here that in majority of these studies, the parent structure was considered to be a classical group rather than an $L$(fuzzy)-group. This setting has a significant limitation that it does not allow  the formulation of various concepts from classical group theory to the fuzzy (lattice valued)  group theory. This drawback can be easily removed by taking the parent structure to be a $L$(fuzzy)-group rather than an ordinary group. Indeed, in \cite{ajmal_gen, ajmal_char, ajmal_nc, ajmal_nil, ajmal_nor, ajmal_sol}, Ajmal and Jahan have introduced and studied various algebraic structures in $L$-setting specifically keeping in view their compatibility. Additionally, the authors of this paper have continued this research in \cite{jahan_max, jahan_app} by introducing the maximal and Frattini $L$-subgroups of an $L$-group and exploring their relationship with various structures in $L$-group theory. This paper is a continuation of similar studies. 

The notion of the conjugate of a fuzzy group was introduced by Mukherjee and Bhattacharya \cite{mukherjee_some} in 1986. 
A similar definition has also been given in \cite{shukla_con} in 2013.  However, in their definition the notion of conjugate fuzzy subgroup of an ordinary group has been discussed rather than conjugate fuzzy subgroups of a fuzzy group. The conjugate fuzzy subgroup has been defined with respect to an element of the classical group instead of the fuzzy group itself. In \cite{ajmal_nil}, Ajmal and Jahan have provided a new definition of the normalizer of an $L$(fuzzy)-subgroup of an $L$-subgroup that is compatible with the corresponding notions in classical group theory. It is well known that the concept of conjugate subgroups and normalizers have important relationship in classical group theory. We wish to explore such relationship in the studies of $L$(fuzzy)-groups as well. Here we note that notion of conjugate fuzzy subgroup given in  \cite{mukherjee_some, shukla_con} cannot be used to extend several significant concepts to the $L$(fuzzy) setting such as the notions of pronormal and abnormal subgroups, etc. We remove these limitations by firstly taking the parent structure to be an $L$-group $\mu$ and secondly by defining the conjugate of the $L$-subgroup by an $L$-point in $\mu$ rather than an element of the ordinary group $G$. Here, we remark that the definition of the conjugate of an $L$-subset $\eta$ by another $L$-subset $\theta$ has been explored in \cite{jahan_nil}.

We begin our work in section 3 by introducing the notion of the conjugate of an $L$-subgroup by an $L$-point in the parent $L$-group $\mu$. Moreover, we show that the conjugate $L$-subset so formed defines an $L$-subgroup of $\mu$. This notion of conjugate $L$-subgroup  has been illustrated through an example. Next, It has been shown that the image as well as the pre-image of a conjugate $L$-subgroup of an $L$-group under a group homomorphism are themselves conjugate $L$-subgroups. Then, a level subset characterization for conjugate $L$-subgroups has been developed. We end the section by exploring the properties of maximality under conjugacy. 

In section 4, we investigate the relationship between the notions of normality and conjugacy of $L$-subgroups. Firstly, we show that $\eta$ is a normal $L$-subgroup of $\mu$ if and only if every conjugate $L$-subgroup of $\eta$ is contained in $\eta$. Moreover, the equality holds if tip of $\eta$ is equal to the tip of the conjugate $L$-subgroup. Next, we exhibit a significant relationship between the normalizer of a conjugate $L$-subgroup, $N(\eta^{a_z})$, and the conjugate of the normalizer of the $L$-subgroup, $N(\eta)^{a_z}$. This has also been demonstrated with the help of an example. Next, we provide a new definition of the normalizer of an $L$-subgroup using the notion of conjugacy developed in this paper. We conclude the section by exhibiting this definition by an example.   

\section{Preliminaries}

Throughout this paper, $L = \langle L, \leq, \vee, \wedge  \rangle$ denotes a completely distributive lattice where '$\leq$' denotes the partial ordering on $L$ and '$\vee$'and '$\wedge$' denote, respectively, the join (supremum) and meet (infimum) of the elements of $L$. Moreover, the maximal and minimal elements of $L$ will be denoted by $1$ and $0$, respectively. The concept of completely distributive lattice can be found in any standard text on the subject \cite{gratzer_lattices}. 

The notion of a fuzzy subset of a set was introduced by Zadeh \cite{zadeh_fuzzy} in 1965. In 1967, Goguen \cite{goguen_sets} extended this concept to $L$-fuzzy sets. In this section, we recall the basic definitions and results associated with $L$-subsets that shall be used throughout this work. These definitions can be found in chapter 1 of \cite{mordeson_comm}.

Let $X$ be a non-empty set. An $L$-subset of $X$ is a function from $X$  into $L$. The set of  $L$-subsets of $X$ is called  the $L$-power set of $X$ and is denoted by $L^X$.  For  $\mu \in L^X, $  the set $ \lbrace\mu(x) \mid x \in X \rbrace$  is called the image of $\mu$  and is denoted by  Im $\mu $. The tip and tail of $ \mu $  are defined as $\bigvee \limits_{x \in X}\mu(x)$ and $\bigwedge \limits_{x \in X}\mu(x)$, respectively. An $L$-subset $\mu$ of $X$ is said to be contained in an $L$-subset $\eta$  of $X$ if  $\mu(x)\leq \eta (x)$ for all $x \in X$. This is denoted by $\mu \subseteq \eta $.  For a family $\lbrace\mu_{i} \mid i \in I \rbrace$  of $L$-subsets in  $X$, where $I$  is a non-empty index set, the union $\bigcup\limits_{i \in I} \mu_{i} $    and the intersection  $\bigcap\limits_{i \in I} \mu_{i} $ of  $\lbrace\mu_{i} \mid i \in I \rbrace$ are, respectively, defined by
\begin{center}
	$\bigcup\limits_{i \in I} \mu_{i}(x)= \bigvee\limits_{i \in I} \mu(x)  $ \quad and \quad $\bigcap\limits_{i \in I} \mu_{i} (x)= \bigwedge\limits_{i \in I} \mu(x) $
\end{center}
for each  $x \in X $. If  $\mu \in L^X $  and  $a \in L $,  then the level  subset $\mu_{a}$ of $\mu$  is defined as
\begin{center}
	$\mu_{a}= \lbrace x \in X \mid \mu (x) \geq a\rbrace.$
\end{center}
 For $\mu, \nu \in L^{X} $, it can be verified easily that if $\mu\subseteq \nu$, then $\mu_{a} \subseteq \nu_{a} $ for each $a\in L $.

For $a\in L$ and $x \in X$, we define $a_{x} \in L^{X} $ as follows: for all $y \in X$,
\[
a_{x} ( y ) =
\begin{cases}
	a &\text{if} \ y = x,\\
	0 &\text{if} \ y\ne x.
\end{cases}
\]
$a_{x} $ is referred to as an $L$-point or $L$-singleton. We say that $a_{x} $ is an $L$-point of $\mu$ if and only if
$\mu( x )\ge a$ and we write $a_{x} \in \mu$. 

Let $S$ be a groupoid. The set product $\mu \circ \eta$   of $\mu, \eta \in L^S$ is an $L$-subset of $S$ defined by
\begin{center}
	$\mu \circ \eta (x) = \bigvee \limits_{x=yz}\lbrace\mu (y) \wedge \eta (z) \rbrace.$
\end{center}

\noindent Note that if $x$ cannot be factored as  $x=yz$  in $S$, then  $\mu \circ \eta (x)$, being  the least upper bound of the empty set, is zero. It can be verified that the set product is associative in  $L^S$  if $S$ is a semigroup.

Let $f$ be a mapping from a set $X$ to a set $Y$. If $\mu \in L ^{X}$ and $\nu \in L^{Y}$, then the image $f(\mu )$
of $\mu $ under $f$ and the preimage $f^{-1} (\nu )$ of $\nu $ under $f$ are $L$-subsets of $Y$ and $X$ respectively, defined by
\[
f(\mu )(y)=\bigvee\limits_{x\in f^{-1} (y)} \{\mu (x)\} \text{ and } f^{-1} (\nu )(x)=\nu (f(x)).
\]
Again,  if $f^{-1} (y)=\phi $,
then $f(\mu )(y)$ being the least upper bound of the empty set, is zero.

Throughout this paper, $G$ denotes an ordinary group with the identity element `$e$' and $I$ denotes a non-empty indexing set. Also, $1_A$ denotes the characteristic function of a non-empty set $A$.

In 1971, Rosenfeld \cite{rosenfeld_fuzzy} applied the notion of fuzzy sets to groups in order to introduce the fuzzy subgroup of a group. Liu \cite{liu_op}, in 1981, extended the notion of fuzzy subgroups to $L$-fuzzy subgroups ($L$-subgroups), which we define below.   

\begin{definition}
	Let $\mu \in L ^G $. Then, $\mu $ is called an $L$-subgroup of $G$ if for each $x, y\in G$,
	\begin{enumerate}
		\item[({i})] $\mu (xy)\ge \mu (x)\wedge \mu (y)$,
		\item[({ii})] $\mu (x^{-1} )=\mu (x)$.
	\end{enumerate}
	The set of $L$-subgroups of $G$ is denoted by $L(G)$. Clearly, the tip of an $L$-subgroup
	is attained at the identity element of $G$.
\end{definition}

\begin{theorem}(\cite{mordeson_comm}, Lemma 1.2.5)
	\label{lev_gp}
	Let $\mu \in L ^G $. Then, $\mu $ is an $L$-subgroup of $G$ if and only if each non-empty level subset $\mu_{a} $ is a subgroup of $G$.
\end{theorem}

\begin{theorem}(\cite{mordeson_comm}, Theorems 1.2.10, 1.2.11)
	\label{hom_gp}
	Let $f : G \rightarrow H$ be a group homomorphism. Let $\mu \in L(G)$ and $\nu \in L(H)$. Then, $f(\mu) \in L(H)$ and $f^{-1}(\nu) \in L(G)$.
\end{theorem}

It is well known in literature that the intersection of an arbitrary family of $L$-subgroups of a group is an $L$-subgroup of the given group.

The concept of normal fuzzy subgroup of a group was introduced by Liu \cite{liu_inv} in 1982. We define the normal $L$-subgroup of a group $G$ below:  

\begin{definition}
	Let $\mu\in L(G)$. Then, $\mu $ is called a normal $L$-subgroup of $G$ if for all $x, y \in  G$, $\mu ( xy ) = \mu ( yx )$.
\end{definition}

\noindent The set of normal $L$-subgroups of $G$ is denoted by $NL(G)$. 

\begin{theorem}(\cite{mordeson_comm}, Theorem 1.3.3)
	\label{lev_norgp}
	Let $\mu \in L{(G)}$. Then, $\mu \in NL(G)$ if and only if each non-empty level subset $\mu_a$ is a normal subgroup of $G$.	
\end{theorem}

Let $\eta, \mu\in L^{{G}}$ such that $\eta\subseteq\mu$. Then, $\eta$ is said to be an $L$-subset of $\mu$. The set of all $L$-subsets of $\mu$ is denoted by $L^{\mu}.$
Moreover, if $\eta,\mu\in L(G)$ such that  $\eta\subseteq \mu$, then $\eta$ is said to be an $L$-subgroup of $\mu$. The set of all $L$-subgroups of $\mu$ is denoted by $L(\mu)$. It is well known that the intersection of an arbitray family of $L$-subgroup of an $L$-group $\mu$ is again an $L$-subgroup of $\mu$.
 \begin{definition}(\cite{ajmal_char})
 	Let $\eta \in L ^{\mu}$.  Then, the $L$-subgroup of $\mu$ generated by $\eta $ is defined as the smallest $L$-subgroup of $\mu$
 	which contains $\eta $. It is denoted by $\langle \eta \rangle $, that is,
 	\[
 	\langle \eta \rangle =\cap\{\eta _{{i}} \in L(\mu) \mid \eta \subseteq \eta _{i}\}.
 	\]
 \end{definition}

From now onwards, $\mu$ denotes an $L$-subgroup of $G$ which shall be considered as the parent $L$-group. 

\begin{definition}(\cite{ajmal_sol}) 
	Let $\eta\in L(\mu)$ such that $\eta$ is non-constant and $\eta\ne\mu$. Then, $\eta$ is said to be a proper $L$-subgroup of $\mu$.
\end{definition}

\noindent Clearly, $\eta$ is a proper $L$-subgroup of $\mu$ if and only if $\eta$ has distinct tip and tail and $\eta\ne\mu$.

\begin{definition}(\cite{ajmal_nil})
	Let $\eta \in L(\mu)$. Let $a_0$ and $t_0$ denote the tip and tail of $\eta$, respectively. We define the trivial $L$-subgroup of $\eta$ as follows:
	\[ \eta_{t_0}^{a_0}(x) = \begin{cases}
		a_0 & \text{if } x=e,\\
		t_0 & \text{if } x \neq e.
	\end{cases} \]
\end{definition}

\begin{theorem}(\cite{ajmal_nil}, Theorem 2.1)
	\label{lev_sgp}
	Let $\eta \in L^\mu$. Then, $\eta\in L(\mu)$ if and only if each non-empty level subset $\eta_a$  is a subgroup of $\mu_a$.
		
\end{theorem}

The normal fuzzy subgroup of a fuzzy group was introduced by Wu \cite{wu_normal} in 1981. We note that for the development of this concept, Wu\cite{wu_normal} preferred  $L$-setting. Below, we recall the notion of a normal $L$-subgroup of an $L$-group:

\begin{definition}(\cite{wu_normal})
	Let $\eta \in L(\mu)$. Then, we say that  $\eta$  is a normal $L$-subgroup of $\mu$   if  
	\begin{center}
		$\eta(yxy^{-1}) \geq \eta(x)\wedge \mu(y)$ for  all  $x,y \in G.$
	\end{center}
\end{definition}

\noindent The set of normal $L$-subgroups of $\mu$  is denoted by $NL(\mu)$. If $\eta \in NL(\mu)$, then we write\vspace{.2cm} $ \eta \triangleleft \mu$. 

Here, we mention that the arbitrary intersection of a family of normal $L$-subgroups of an $L$-group $\mu$ is again a normal $L$-subgroup of $\mu$. 

\begin{theorem}(\cite{mordeson_comm}, Theorem 1.4.3)
	\label{lev_norsgp}
	Let $\eta \in L(\mu)$. Then, $\eta\in NL(\mu)$ if and only if each non-empty level subset $\eta_a$  is a normal subgroup of $\mu_a$.
\end{theorem}

Lastly, recall the following form \cite{ajmal_gen}:

\begin{theorem}(\cite{ajmal_gen}, Theorem 3.1)
	\label{gen}
	Let $\eta\in L^{^{\mu}}.$ Let $a_{0}=\mathop {\vee}\limits_{x\in G}{\left\{\eta\left(x\right)\right\}}$ and define an $L$-subset $\hat{\eta}$ of $G$ by
	\begin{center}
		$\hat{\eta}\left(x\right)=\mathop{\vee}\limits_{a \leq a_{0}}{\left\{a \mid x\in\left\langle \eta_{a}\right\rangle\right\}}$.
	\end{center}
	
	\noindent Then, $\hat{\eta}\in L(\mu)$ and  $\hat{\eta} =\left\langle \eta \right\rangle$. Moreover, tip~$\langle \eta\rangle=\text{tip~}\eta$.
\end{theorem}

\section{Conjugate $L$-subgroups}
	The notion of conjugate subgroups has played an important role in the evolution of classical group theory. We are motivated to contemplate such progress in $L$-group theory. For this, we firstly recall the definition of conjugate   of an $L$-subset by an $L$-subset from \cite{jahan_nil}:

\begin{definition}\label{dfncon}(\cite{jahan_nil})
	Let $\eta, \theta \in L^{\mu}$. Define an $L$-subset $\theta\eta\theta^{{-1} } $ of $G$ as follows:
	\[
	\theta\eta\theta^{{-1} } ( x ) = \bigvee\limits _{x = z y z^{-1} } \{ \eta ( y ) \wedge \theta ( z ) \}\quad\text{for each} \ x \in G.
	\]
	We call $\theta\eta\theta^{{-1} }$ the conjugate of $\eta$ by $\theta$. Clearly, $\theta\eta\theta^{{-1} } \subseteq \mu.$ Hence the $L$-subgroup $\langle\theta\eta\theta^{{-1} }\rangle \in L(\mu) $ and is denoted by $\eta ^ \theta.$ If $\eta$ and $\theta $ are $L$-subgroups of $\mu$, then clearly, in view of Theorem \ref{gen}, the tip of the $L$-subgroup $\eta^\theta$  is  $\eta^{\theta}(e)=\eta (e)\wedge \theta(e)$.
\end{definition}

\noindent The notion of conjugate fuzzy subgroups has appeared in \cite{shukla_con} in the year 2013. However, in this definition the notion of conjugate fuzzy subgroup of an ordinary group has been discussed rather than conjugate fuzzy subgroups of a fuzzy group. Below, we introduce the notion of conjugate  $L$-subgroups of an $L$-group. Firstly, observe that in view of the Definition \ref{dfncon}, the conjugate of an $L$-subgroup $\eta$ by an $L$-point $a_z \in \mu$ is  given by:
 \[ \eta^{a_z}(x) = a \wedge \eta(zxz^{-1}) \quad \text{ for all } x \in G. \]	
\noindent Here, we show that the conjugate of an $L$-subgroup by an $L$-point in $\mu$ is an $L$-subgroup of $\mu$. 

\begin{theorem}
	Let $\eta \in L(\mu)$ and $a_z$ be an $L$-point of $\mu$. Then, $\eta^{a_z}$ is an $L$-subgroup of $\mu$.
\end{theorem}
\begin{proof}
	Let $x,y \in G$. Then,
	\begin{equation*}
		\begin{split}
			\eta^{a_z}(xy) &= a \wedge \eta(z(xy)z^{-1}) \\
			&= a \wedge \eta((zxz^{-1})(zyz^{-1})) \\
			&\geq a \wedge \eta(zxz^{-1}) \wedge \eta(zyz^{-1}) \qquad \qquad (\text{since $\eta$ is an $L$-subgroup of $\mu$}) \\
			&= (a \wedge \eta(zxz^{-1})) \wedge (a \wedge \eta(zyz^{-1})) \\
			&= \eta^{a_z}(x) \wedge \eta^{a_z}(y).  
		\end{split}
	\end{equation*}
	Similarly, we can verify that
	$
	\eta^{a_z}(x^{-1}) = \eta^{a_z}(x).  
	$
	Thus $\eta^{a_z} \in L(G).$ Now, to show that $\eta^{a_z} \subseteq \mu$, let $x \in G$. Then,
	\begin{equation*}
	\begin{split}
		\mu(x) &= \mu(z^{-1}(zxz^{-1})z) \\
		&\geq \mu(z) \wedge \mu(zxz^{-1}) \\
		&\geq a \wedge \eta(zxz^{-1}) \qquad \qquad\qquad (\text{since $a_z \in \mu$ and $\eta \subseteq \mu$}) \\
		&= \eta^{a_z}(x).
	\end{split}
	\end{equation*}
This proves that $\eta^{a_z} \in L(\mu)$.
\end{proof}

\begin{remark}\label{conj_tip}
	Clearly, $\text{tip}(\eta^{a_z}$) = $a \wedge \text{tip}(\eta)$, since
	$ \eta^{a_z}(e) = a \wedge \eta(zez^{-1}) = a \wedge \eta(e).$
\end{remark}
\noindent We demonstrate the notion of conjugate $L$-subgroups with the following example:

\begin{example}
	Let 
	\[ M=\{ l,f,a,b,c,d,u \} \]
	be the lattice given by Figure 1. Let $G=S_4$, the group of all permutations of the set $\{1,2,3,4\}$ with the identity element $e$. Let
	\[ D_4^1 = \langle (24), (1234) \rangle, D_4^2 = \langle (12), (1324) \rangle, D_4^3 = \langle (23), (1342) \rangle \]
	denote the dihedral subgroups of $G$ and $V_4 = \{e, (12)(34), (13)(24), (14)(23)\}$ denote the Klein-4 subgroup of $G$. 
	\begin{center}
		\includegraphics[scale=.7]{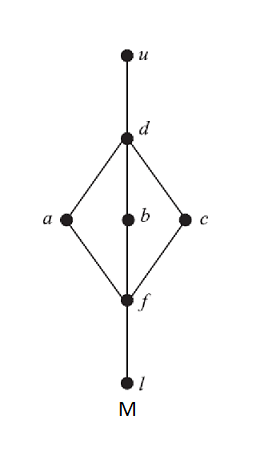}\\
		\centering \text{Figure 1}		
	\end{center}
	Define the $L$-subgroup $\mu$ of $G$ as follows:
	\[ \mu(x) = \begin{cases}
		u &\text{if } x \in V_4,\\
		d &\text{if } x \in S_4 \setminus V_4.
	\end{cases} \]
	Next, let $\eta$ be the $L$-subset of $\mu$ be defined by 
	\[ \eta(x) = \begin{cases}
		u &\text{if } x=e,\\
		d &\text{if } x \in V_4 \setminus \{e\},\\
		a &\text{if } x \in D_4^1 \setminus V_4,\\
		b &\text{if } x \in D_4^2 \setminus V_4,\\
		c &\text{if } x \in D_4^3 \setminus V_4,\\
		l &\text{if } x \in S_4 \setminus \mathop{\cup}\limits_{i=1}^3 D_4^i.
	\end{cases} \]
	Since each level subset $\eta_t$ is a subgroup of $\mu_t$ for all $t \leq u$, $\eta$ is an $L$-subgroup of $\mu$. Note that $d_{(123)} \in \mu$. We determine $\eta^{d_{(123)}}$.
	
	\noindent By definition, 
	\[ \eta^{d_{(123)}}(x) = d \wedge \eta((123)x(132)) \]
	for all $x \in S_4$. Hence
	\[	\eta^{d_{(123)}}(x) = \begin{cases}
			d \wedge u &\text{if } x=e,\\
			d \wedge d &\text{if } x \in V_4 \setminus \{e\},\\
			d \wedge b &\text{if } x \in D_4^1 \setminus V_4,\\
			d \wedge c &\text{if } x \in D_4^2 \setminus V_4,\\
			d \wedge a &\text{if } x \in D_4^3 \setminus V_4,\\
			d \wedge l &\text{if } x \in S_4 \setminus \mathop{\cup}\limits_{i=1}^3 D_4^i.
		\end{cases} \]
	Thus
	\[\eta^{d_{(123)}}(x) = \begin{cases}
		d &\text{if } x \in V_4,\\
		b &\text{if } x \in D_4^1 \setminus V_4,\\
		c &\text{if } x \in D_4^2 \setminus V_4,\\
		a &\text{if } x \in D_4^3 \setminus V_4,\\
		l &\text{if } x \in S_4 \setminus \mathop{\cup}\limits_{i=1}^3 D_4^i.
	\end{cases} \] 
\end{example}
\vspace{.2cm}
\noindent Below we discuss the set product of conjugate $L$-subgroups:
\begin{theorem}
	Let $\eta, \nu \in L(\mu)$ and $a_z$ be an $L$-point of $\mu$. Then, 
	\[ (\eta \circ \nu)^{a_z} = \eta^{a_z} \circ \nu^{a_z}. \]
\end{theorem}
\begin{proof}
	Let $g \in G$. Then,
	\begin{equation*}
	\begin{split}
		(\eta^{a_z} \circ \nu^{a_z})(g) &= \vee \left\{\eta^{a_z}(x) \wedge \nu^{a_z}(y) \mid xy = g \right\} \\
			&= \vee \left\{ \{ a \wedge \eta(zxz^{-1}) \} \wedge \{ a \wedge \nu(zyz^{-1})\} \mid xy=g \right\} \\
			&= a \wedge \left\{ \vee \left\{ \eta(zxz^{-1}) \wedge \nu(zyz^{-1}) \mid (zxz^{-1})(zyz^{-1}) = zgz^{-1} \right\} \right\} \\
			&= a \wedge \left\{ \vee \left\{ \eta(x_1) \wedge \nu(y_1) \mid x_1y_1 = zgz^{-1} \right\} \right\} \\
			&= a \wedge (\eta \circ \nu)(zgz^{-1}) \\
			&= (\eta \circ \nu)^{a_z}(g).			
	\end{split}
	\end{equation*}
	Hence the result.
\end{proof}

\noindent In Theorems \ref{hom_conj1} and \ref{hom_conj2}, we study the properties of conjugate $L$-subgroups under group homomorphisms.

\begin{theorem}\label{hom_conj1}
	Let $f : G \rightarrow H$ be a group homomorphism and $\mu \in L(G)$. Then, for $\eta \in L(\mu)$ and $a_z \in \mu$, the $L$-subgroup $f(\eta^{a_z})$ is a conjugate $L$-subgroup of $f(\eta)$ in $f(\mu)$. In fact, 
	\[ f(\eta^{a_z}) = f(\eta)^{a_{f(z)}}. \]
\end{theorem}
\begin{proof}
	Firstly, note that $a_{f(z)} \in f(\mu)$, since
	\[ (f(\mu))(f(z)) = \vee\{ \mu(x) \mid f(x) = f(z) \} \geq \mu(z) \geq a. \] 
	Next, let $y \in H$. Then,
	\begin{equation*}
		\begin{split}
			f(\eta^{a_z})(y) &= \vee \{ \eta^{a_z}(x) \mid f(x) = y \} \\
			&= \vee \{ a \wedge \eta(zxz^{-1}) \mid f(x) = y \} \\
			&= \vee \{ a \wedge \eta(zxz^{-1}) \mid f(zxz^{-1}) = f(z)yf(z)^{-1} \}\\
			&= a \wedge \{\vee \{ \eta(u) \mid f(u) = f(z)yf(z)^{-1} \} \}~~(\text{since~} L \text{~is a completely distributive lattice~}) \\
			&= a \wedge f(\eta)(f(z)yf(z)^{-1}) \\
			&= f(\eta)^{a_{f(z)}}(y).
		\end{split}
	\end{equation*}
	Hence the result.
\end{proof}

\begin{theorem}\label{hom_conj2}
	Let $f: G \rightarrow H$ be a surjective group homomorphism and let $\mu \in L(H)$. Then, for $\eta \in L(\mu)$ and $a_z \in \mu$, the $L$-subgroup $f^{-1}(\eta^{a_z})$ is a conjugate $L$-subgroup of $f^{-1}(\eta)$ in $f^{-1}(\mu)$. In fact, 
	\[ f^{-1}(\eta^{a_z}) = f^{-1}(\eta)^{a_s},  \]
	where $s \in f^{-1}(z)$. 
\end{theorem}
\begin{proof}
	Let $s \in f^{-1}(z)$. Firstly, note that $a_s \in f^{-1}(\mu)$, since
	\[ f^{-1}(\mu)(s) = \mu(f(s)) = \mu(z) \geq a. \]
	Next, let $x \in G$. Then, 
	\begin{equation*}
		\begin{split}
			f^{-1}(\eta^{a_z})(x) &= \eta^{a_z}(f(x))\\
			&= a \wedge \eta(zf(x)z^{-1}) \\
			&= a \wedge \eta(f(s) f(x) f(s)^{-1})\\
			&= a \wedge \eta(f(sxs^{-1})) \\
			&= a \wedge f^{-1}(\eta)(sxs^{-1}) \\
			&= f^{-1}(\eta)^{a_s}(x).
		\end{split}
	\end{equation*}
This proves that $f^{-1}(\eta^{a_z}) = f^{-1}(\eta)^{a_s}$.
\end{proof}

\noindent Below, we provide a level subset characterization for conjugate $L$-subgroups.

\begin{theorem}\label{lvl_conj}
	Let $\eta, \nu \in L(\mu)$ and $a \in L$ such that $\text{tip}(\nu) = a \wedge \text{tip}(\eta)$. Then, $\nu=\eta^{a_z}$ for $a_z \in \mu$ if and only if $\nu_t = {\eta_t}^{z^{-1}}$ for all $t \leq \text{tip}(\nu)$.
\end{theorem}
\begin{proof}
	($\Rightarrow$) Let $\nu=\eta^{a_z}$. Let $t \leq \text{tip~}{(\nu)} = a \wedge \text{tip}(\eta)$ and let $x \in \nu_t$. Then, $\nu(x) \geq t$, that is, 
	\[ \eta^{a_z}(x) = a \wedge \eta(zxz^{-1}) \geq a \wedge \eta(e) = \text{tip~}(\nu)\geq t. \]
	This implies $\eta(zxz^{-1}) \geq t$, that is, $zxz^{-1} \in \eta_t$. Thus $x \in {\eta_t}^{z^{-1}}$. 
	Hence $\nu_t \subseteq {\eta_t}^{z^{-1}}.$
	To show the reverse inclusion, let $x \in {\eta_t}^{z^{-1}}$. Then, $x = z^{-1}gz$ for some $g \in \eta_t$.	This implies $zxz^{-1} = g \in \eta_t$, that is, $\eta(zxz^{-1}) \geq t$. Moreover, by assumption, $t \leq a$. Hence
	\[ \nu(x) = a \wedge \eta(zxz^{-1}) \geq t. \]
	Thus $x \in \nu_t$ and we conclude that $\nu_t = {\eta^{z^{-1}}_t}$.
	
	\noindent ($\Leftarrow$) Suppose that $\nu_t = {\eta_t}^{z^{-1}}$ for all $t \leq \text{tip}(\nu)$. Let $x \in G$ and let $\nu(x) = b$. Then, $b \leq \text{tip}(\nu)$ and by the hypothesis, $\nu_b = {\eta_b}^{z^{-1}}$. Thus $x \in {\eta_b}^{z^{-1}}$, that is, $x = z^{-1}gz$ for some $g \in \eta_b$. This implies	$zxz^{-1} = g \in \eta_b$, and hence $\eta(zxz^{-1}) \geq b = \nu(x)$. Thus 
	\[ \eta^{a_z}(x) = a \wedge \eta(zxz^{-1}) \geq a \wedge \nu(x) = \nu(x) \]
	and we conclude that $\nu \subseteq \eta^{a_z}$. For the reverse inclusion, let $x \in G$ and let $b = a \wedge \eta(zxz^{-1}) \leq a \wedge \eta(e)=\text{tip~}\nu$. Then, by hypothesis, $\nu_b = {\eta_b}^{z^{-1}}$, or equivalently, $\eta_b = {\nu_b}^z$. Now, $zxz^{-1} \in \eta_b = {\nu_b}^z$ implies	$zxz^{-1} = zgz^{-1}$ for some $g \in \nu_b$. Thus $x = g \in \nu_b$ and hence 
	\[ \nu(x) \geq b = a \wedge \eta(zxz^{-1}) = \eta^{a_z}(x). \]
	Thus $\eta^{a_z} \subseteq \nu$ and we conclude that $\nu = \eta^{a_z}$.
\end{proof}

\begin{theorem}
	Let $H$ and $K$ be subgroups of $G$. Then, $K$ is conjugate to $H$ in $G$ if and only if $1_K$ is conjugate to $1_H$ as  $L$-subgroups of $1_G$.
\end{theorem}
\begin{proof}
$(\Rightarrow)$ Since $H$ and $K$ are conjugate in $G$,  there exists $z \in G$ such that $ K = H^z$.
	Since $z \in G$, $1_{z^{-1}} \in 1_G$. We claim that $1_K = ({1_H})^{1_{z^{-1}}}$.
	
	\noindent  Let $x \in G$. If $x \notin K$, then $1_K(x) = 0 \leq ({1_H})^{1_{z^{-1}}}(x)$.
	If $x \in K$, then $x=zhz^{-1}$ for some $h \in H$. This implies \[ ({1_H})^{1_{z^{-1}}}(x) = 1 \wedge 1_H(z^{-1}xz) = 1=1_K(x). \]
	So we conclude that $1_K \subseteq (1_H)^{1_{z^{-1}}}$. Similarly we can prove that $(1_H)^{1_{z^{-1}}} \subseteq 1_K$.
	
\noindent ($\Leftarrow$) Since $1_H$ and $1_K$ are conjugate $L$-subgroups in $1_G$,  there exists $a_z \in 1_G$ such that $1_K = {1_H}^{a_z}$.
	Note that $a > 0$, for if $a=0$, then $1_K(x) = a \wedge 1_H(zxz^{-1}) = 0$ for all $x \in G$, which contradicts the fact that $1_K(e) = 1$.
	\noindent We claim that $K = H^{z^{-1}}$. 
	
	\noindent Let $x \in K$. Then, $1_K(x) = 1$. This implies
	\[ 1 = {1_H}^{a_z}(x) = a \wedge {1_H}(zxz^{-1}). \]
	Thus $a = 1$ and $1_H(zxz^{-1}) = 1$. Hence $zxz^{-1} \in H$, that is, $x \in H^{z^{-1}}$. Consequently, $K \subseteq H^{z^{-1}}$.	For the reverse inclusion, let $x \in H^{z^{-1}}$. This implies,
	\[ {1_H}^{a_z}(x) = a \wedge 1_H(zxz^{-1}) = a. \]
	Thus $1_K(x) = a > 0$. It follows that $1_K(x) = 1$, that is, $x \in K$. Therefore, $ H^{z^{-1}} \subseteq K$. Hence the result.
\end{proof}

\noindent The notion of a maximal $L$-subgroup of an $L$-group has been introduced in \cite{jahan_max}. We recall the definition below: 

\begin{definition}(\cite{jahan_max})
	Let $\mu \in L(G)$. A proper $L$-subgroup $\eta$ of $\mu$ is said to be a maximal $L$-subgroup of $\mu$ if, whenever $\eta \subseteq \theta \subseteq \mu$ for some $\theta \in L(\mu)$, then either $\theta = \eta$ or $\theta = \mu$.
\end{definition}

\noindent Here, we discuss conjugate of a maximal $L$-subgroup of an $L$-group.

\begin{theorem}
Let $L$ be a chain. Let $\eta$ be a maximal $L$-subgroup of an $L$-group $\mu$ and $a_z$ be an $L$-point of $\mu$. Then, either $\eta^{a_z}=\mu^{a_z}$ or $\eta^{a_z}$ is a maximal $L$-subgroup of $\mu^{a_z}$. 
\end{theorem}
\begin{proof}
If $\eta^{a_z}=\mu^{a_z}$, then there is nothing to prove. So, let $\eta^{a_z} \subsetneq \mu^{a_z}$. Suppose, if possible, that $\eta^{a_z}$ is not a maximal $L$-subgroup of $\mu^{a_z}$. Then, there exists an $L$-subgroup $\theta$ of $\mu^{a_z}$ such that 
\begin{equation}
\eta^{a_z}\subsetneq \theta \subsetneq \mu^{a_z}.
\end{equation}
\noindent We shall construct an $L$-subgroup $\gamma$ of $\mu$ such that $\eta \subsetneq \gamma \subsetneq \mu$ which will contradict maximality of $\eta$. Define $\gamma ~: ~G \to L$  as follows:
\[ \gamma(x)= \eta(x)\vee \theta(z^{-1}xz)~\text{for all}~ x\in G.\]
 Firstly, we claim that $\eta\subseteq \gamma \subseteq \mu$. Clearly, $\eta \subseteq \gamma$. Next,  as $\theta \subsetneq \mu^{a_z}$, it follows that for all $x \in G$, $\theta (z^{-1}xz)\leq \mu^{a_z}(z^{-1}xz) = a \wedge \mu(x) \leq \mu(x)$. Also,  $\eta(x)\leq\mu(x)$. Consequently,
 \[
 \gamma(x)=\eta(x)\vee \theta (z^{-1}xz)\leq \mu(x) \text{ for all } x \in G.
 \]
 This proves the claim. Now, we shall show that $\gamma \in L(\mu)$. So let $x, y \in G$ and consider
\begin{eqnarray*}
	\gamma(xy)&=&\eta(xy)\vee \theta(z^{-1}xyz)\\
	&=& \eta(xy)\vee \theta((z^{-1}xz)(z^{1}yz))\\
	&\geq& \{\eta(x) \wedge \eta(y)\}\vee \{\theta(z^{-1}xz)\wedge\theta((z^{-1}yz))\}\\
	&\geq& \{\eta(x)\vee \theta(z^{-1}xz)\}\wedge\{\eta(y)\vee \theta(z^{-1}yz)\}\\
	&=& \gamma(x)\wedge \gamma(y).
\end{eqnarray*}
Similarly, we can verify that $\gamma(x^{-1})=\gamma(x)$. Thus $\gamma \in L(\mu)$. Next, we claim that 
\[\eta \subsetneq \gamma \subsetneq \mu.\]
In view of  (1), there exist $x_1, x_2 \in G$ such that 
\[ \eta^{a_z} (x_1) < \theta (x_1)~\text{and}~\theta (x_2) < \mu^{a_z}(x_2). \]
This implies
\begin{equation}
a\wedge\eta (zx_1z^{-1}) < \theta (x_1)~\text{and}~\theta(x_2)< a \wedge \mu(zx_2z^{-1})\leq a.
\end{equation}
Note that by (1), $\text{tip}~\eta^{a_z} =a \wedge \eta(e) \leq \text{tip}~\theta= \theta(e) \leq \text{tip}~ \mu^{a_z}= a \wedge \mu(e) \leq a$. Hence $\eta^{a_z}(x_1) < \theta(x_1) \leq \theta(e) \leq a$ implies $a \wedge \eta(zx_1z^{-1}) < a$. As $L$ is a chain, we must have $a\wedge\eta (zx_1z^{-1})=\eta (zx_1z^{-1})$.  Hence by (2), 
\begin{equation}
\eta (zx_1z^{-1}) < \theta (x_1).
\end{equation}
Now, consider
\begin{eqnarray*}
	\gamma(zx_1z^{-1})&=& \eta(zx_1z^{-1}) \vee \theta(z^{-1}(zx_1z^{-1})z)\\
	&=&\eta (zx_1z^{-1}) \vee \theta(x_1)\\
	&=&\theta(x_1) \qquad\qquad\qquad\qquad\qquad\qquad (\text{by}~(3))\\
	&>& \eta (zx_1z^{-1}).
\end{eqnarray*}
This implies $\eta \subsetneq \gamma$. Further, by (2), we have 
\begin{equation}
\theta(x_2)< \mu^{a_z}(x_2)=a \wedge \mu(zx_2z^{-1})\leq \mu(zx_2z^{-1}).
\end{equation}
In view of (1)
\begin{equation}\eta^{a_z}(x_2) = a \wedge \eta(zx_2z^{-1}) \leq \theta(x_2).\end{equation}
Since $L $ is a chain, either $a \wedge \eta(zx_2z^{-1})=a$ or $a \wedge \eta(zx_2z^{-1})=\eta(zx_2z^{-1})$. In the first case, we have 
\[
a \leq \theta(x_2)\leq \theta(e)\leq a.
\]
Thus $\theta(x_2)=a$. However, by (4), $\theta(x_2)< \mu^{a_z}(x_2)=a \wedge \mu(zx_2z^{-1}) \leq a$. That is $\theta (x_2) <a$. So there is a contradiction. Hence we must have, $a \wedge \eta(zx_2z^{-1})=\eta(zx_2z^{-1})$.  Now by (5),
\[
\eta(zx_2z^{-1})\leq \theta(x_2).
\]
Therefore, 
\[\gamma(zx_2z^{-1}) = \eta (zx_2z^{-1}) \vee \ \theta(x_2)=\theta(x_2)< \mu(zx_2z^{-1}). \text{~(By using (2))}\]
Consequently, $\gamma \subsetneq \mu$. This completes the proof of the claim. However, this a contradiction to the maximality of $\eta$ in $\mu$. Hence the result.

\end{proof}

\section{Conjugacy and Normality of $L$-subgroups}

In this section, we explore the inter-connections between the concepts of conjugacy and normality of $L$-subgroups. We prove significant results pertaining to the notions of conjugate $L$-subgroups and normalizer \cite{ajmal_nor} and explore the various similarities as well as peculiarities of these concepts compared to their group theoretic counterparts. We end this section by providing a new definition of the normalizer using the concept of conjugacy of $L$-subgroups. Thus this section demonstrates the compatibility of the conjugacy of $L$-subgroups with the several concepts developed so far in the study of $L$-subgroups of $L$-groups.

\begin{proposition}
	\label{nor_conj}
	Let $\eta \in L(\mu)$. Then, $\eta$ is a normal $L$-subgroup of $\mu$ if and only if  $\eta^{a_z} \subseteq \eta$ for every $L$-point $a_z \in \mu$. Moreover, if $\eta \in NL(\mu)$ and $\text{tip}(\eta^{a_z}) = \text{tip}(\eta)$, then $\eta^{a_z} = \eta$. 
\end{proposition}	
\begin{proof} Let $\eta$ be a normal $L$-subgroup of $\mu$ and $a_z$ be an $L$-point of $\mu$. Then, for all $x \in G$,
	\begin{equation*}
		\begin{split}
			\eta(x) &= \eta(z^{-1}(zxz^{-1})z) \\
			&\geq \eta(zxz^{-1}) \wedge \mu(z) \qquad\qquad~~ (\text{since $\eta$ is normal in $\mu$})\\
			&\geq \eta(zxz^{-1}) \wedge a \qquad\qquad \qquad (\text{since $a_z \in \mu$})\\
			&= \eta^{a_z}(x).
		\end{split}
	\end{equation*}
	Hence $\eta^{a_z} \subseteq \eta$. Conversely, suppose that $\eta^{a_z} \subseteq \eta$ for all $L$-points $a_z \in \mu$. Let $x, g \in G$ and let $a = \mu(g)$. Then $a_{g^{-1}} \in \mu$ and by the hypothesis, $\eta^{a_{g^{-1}}} \subseteq \eta$. Thus
	\[ \eta(gxg^{-1}) \geq \eta^{a_{g^{-1}}}(gxg^{-1}) = a \wedge \eta(g^{-1}(gxg^{-1})g) = \mu(g) \wedge \eta(x).  \]
	Therefore $\eta$ is a normal $L$-subgroup of $\mu$.
	
	\noindent Next, let $\eta$ be a normal $L$-subgroup of $\mu$ and $a_z$ be an $L$-point of $\mu$ such that $\text{tip}(\eta^{a_z}) = \text{tip}(\eta)$. Then, by Remark \ref{conj_tip}, $a \geq \text{tip}(\eta)$. Thus for all $x \in G$,
	\begin{equation*}
		\begin{split}
			\eta^{a_z}(x) &= a \wedge \eta(zxz^{-1}) \\
			&\geq a \wedge \eta(x) \wedge \mu(z) \qquad\qquad (\text{since $\eta$ is normal in $\mu$})\\
			&= a \wedge \eta(x)  \qquad\qquad\qquad~~~ (\text{since $a_z \in \mu$}) \\
			&= \eta(x). \qquad\qquad\qquad\qquad~ (\text{since $a \geq \eta(e)$})
		\end{split}
	\end{equation*}
	Hence $\eta \subseteq \eta^{a_z}$ and we conclude that $\eta = \eta^{a_z}$.
\end{proof}

\noindent In \cite{ajmal_nor}, Ajmal and Jahan have introduced the notion of the normalizer of an $L$-subgroup by introducing the coset of an $L$-subgroup with respect to an $L$-point. We recall these concepts below:

\begin{definition}(\cite{ajmal_nor})
	Let $\eta \in L(\mu)$ and let $a_x$ be an $L$-point of $\mu$. The left (respectively, right) coset of $\eta$ in $\mu$ with respect to $a_x$ is defined as the set product $a_x \circ \eta$ ($\eta \circ a_x$).
\end{definition}

\noindent From the definition of set product of two $L$-subsets, it can be easily seen that for all $z \in G$, 
\[ (a_x \circ \eta)(z) = a \wedge \eta(x^{-1}z) \quad \text{ and } \quad (\eta \circ a_x)(z) = a \wedge \eta(zx^{-1}). \]

\begin{definition}(\cite{ajmal_nor})
	Let $\eta \in L(\mu)$. The normalizer of $\eta$ in $\mu$, denoted by $N(\eta$), is the $L$-subgroup defined as follows:
	\[ N(\eta) = \bigcup \left\{ a_x \in \mu \mid a_x \circ \eta = \eta \circ a_x \right\}. \]
	$N(\eta)$ is the largest $L$-subgroup of $\mu$ such that $\eta$ is a normal $L$-subgroup of $N(\eta)$. Also, it has been established in \cite {ajmal_nor} that  $\eta$ is a normal $L$-subgroup of $\mu$ if and only if $N(\eta)=\mu$.
\end{definition}
\noindent Below, we demonstrate the conjugate of the normalizer $N(\eta)$ of the $L$-subgroup $\eta$ :
\begin{theorem}
	Let $\eta \in L(\mu)$ and $a_z$ be an $L$-point of $\mu$. Then, for all $g \in G$,
	\[ N(\eta)^{a_z}(g) = a \wedge N(\eta^{a_z})(g). \]
\end{theorem}
\begin{proof}
	Let $g \in G$. Then, 
	\begin{equation*}
	\begin{split} 
	N(\eta)^{a_z}(g) &= a \wedge N(\eta)(zgz^{-1}) \\
		&= a \wedge \{ \vee \{ b \mid b_{zgz^{-1}} \in \mu \text{ and } b_{zgz^{-1}} \circ \eta = \eta \circ b_{zgz^{-1}} \}\} \\
		&= \vee \{ a \wedge b \mid b_{zgz^{-1}} \in \mu \text{ and } b_{zgz^{-1}} \circ \eta = \eta \circ b_{zgz^{-1}} \},
	\end{split}
	\end{equation*}
	since $L$ is completely distributive. Similarly,
	\[ N(\eta^{a_z})(g) = \vee \{ b \mid b_g \in \mu \text{ and } b_g \circ \eta^{a_z} = \eta^{a_z} \circ b_g \}. \]
	Let 
	\[ L_1 = \{ b \in L \mid b_g \in \mu \text{ and } b_g \circ \eta^{a_z} = \eta^{a_z} \circ b_g \} \]
	and
	\[ L_2 = \{ b \in L \mid b_{zgz^{-1}} \in \mu \text{ and } b_{zgz^{-1}} \circ \eta = \eta \circ b_{zgz^{-1}} \}. \]
	We claim that if  $b \in L_2$, then $a \wedge b \in L_1$. For this, let $b \in L_2$. Then,
	\[ b_{zgz^{-1}} \in \mu \text{ and } b_{zgz^{-1}} \circ \eta = \eta \circ b_{zgz^{-1}}. \]
	Firstly, since $ b_{zgz^{-1}} \in \mu$,	$\mu(zgz^{-1}) \geq b$.	This, and the fact that $a_z \in \mu$, implies 
	\[ \mu(g) = \mu(z^{-1}(zgz^{-1})z) \geq \mu(z) \wedge \mu(zgz^{-1}) \geq a \wedge b. \]
	Hence $(a \wedge b)_g \in \mu$. Next, let $x \in G$. Since	$b_{zgz^{-1}} \circ \eta = \eta \circ b_{zgz^{-1}}$, we have
	\[ (b_{zgz^{-1}} \circ \eta)(x) = (\eta \circ b_{zgz^{-1}})(x), \]
	that is,
	\[ b \wedge \eta((zgz^{-1})^{-1}x) = b \wedge \eta(x(zgz^{-1})^{-1}). \]
	Hence
	\[  b \wedge \eta(zg^{-1}z^{-1}x) = b \wedge \eta(xzg^{-1}z^{-1}). \]
	This implies
	\[ a \wedge (b \wedge \eta(zg^{-1}z^{-1}x)) = a \wedge (b \wedge \eta(xzg^{-1}z^{-1})), \]
	or equivalently
	\begin{equation}{\label{eq_norm1}}
	 (a \wedge b) \wedge (a \wedge \eta(zg^{-1}z^{-1}x)) = (a \wedge b) \wedge (a \wedge \eta(xzg^{-1}z^{-1}))~\text{for all} ~x \in G.
	\end{equation}
	Now, let $y$ be any arbitrary element of $G$. Since equation (\ref{eq_norm1}) holds for all $x \in G$, taking $x = zyz^{-1}$, we get
	\[ (a \wedge b) \wedge (a \wedge \eta(zg^{-1}yz^{-1})) = (a \wedge b) \wedge (a \wedge \eta(zyg^{-1}z^{-1})). \]
	This implies
	\[ (a \wedge b) \wedge \eta^{a_z}(g^{-1}y) = (a \wedge b) \wedge \eta^{a_z}(yg^{-1}). \] 
	Therefore
	\[ ((a \wedge b)_g \circ \eta^{a_z})(y) = (\eta^{a_z} \circ (a \wedge b)_g)(y). \]
	Since, by assumption, $y$ is an arbitrary element of $G$, we conclude that 
	\[ (a \wedge b)_g \circ \eta^{a_z} = \eta^{a_z} \circ (a \wedge b)_g. \]
	Therefore we have shown that 
	\[ (a \wedge b)_g \in \mu \text{ and } (a \wedge b)_g \circ \eta^{z_x} = \eta^{a_z} \circ (a \wedge b)_g. \]
	Hence $(a \wedge b) \in L_1$. Finally,
	\begin{equation*}
	\begin{split}
		N(\eta)^{a_z}(g) &= a \wedge N(\eta)(zgz^{-1}) \\
			&= a \wedge (\vee \{ b \mid b \in L_2 \}) \\
			&= a \wedge \{\vee \{ a \wedge b \mid b \in L_2 \}\} \\
			&\leq a \wedge \{\vee \{ c \mid c \in L_1\}\} \\
			&= a \wedge N(\eta^{a_z})(g).
	\end{split}
	\end{equation*}
	For the reverse inequality, let $b \in L_1$. We show that $a \wedge b \in L_2$. Firstly, since $b \in L_1$, $b_g \in \mu$. Thus
	\[ \mu(zgz^{-1}) \geq \mu(z) \wedge \mu(g) \geq a \wedge b. \]
	Therefore $(a \wedge b)_{zgz^{-1}} \in \mu$. Next, let $x \in G$. Since $b_g \circ \eta^{a_z} = \eta^{a_z} \circ b_g$, we have 
	\[ b \wedge \eta^{a_z}(g^{-1}x) = b \wedge \eta^{a_z}(xg^{-1}), \]
	that is, 
	\[ b \wedge (a \wedge \eta(zg^{-1}xz^{-1})) = b \wedge (a \wedge \eta(zxg^{-1}z^{-1})). \]
	Thus 
	\begin{equation}\label{eq_norm2}
		(a \wedge b) \wedge \eta(zg^{-1}xz^{-1}) = (a \wedge b) \wedge \eta(zxg^{-1}z^{-1}) \text{~for all}~x \in G.
	\end{equation}
	Let $y$ be any arbitrary element of $G$. Since equation (\ref{eq_norm2}) holds for all $x \in G$, taking $x = z^{-1}yz$, we get
	\[ (a \wedge b) \wedge \eta(zg^{-1}z^{-1}y) = (a \wedge b) \wedge \eta(yzg^{-1}z^{-1}), \]
	or equivalently,
	\[((a \wedge b)_{zgz^{-1}} \circ \eta)(y) = (\eta \circ (a \wedge b)_{zgz^{-1}})(y). \]
	Since, by assumption, $y$ is an arbitrary element of $G$, we conclude that 
	\[ (a \wedge b)_{zgz^{-1}} \circ \eta = \eta \circ (a \wedge b)_{zgz^{-1}} \]
	Hence we have shown that 
	\[ (a \wedge b)_{zgz^{-1}} \in \mu \text{ and } (a \wedge b)_{zgz^{-1}} \circ \eta = \eta \circ (a \wedge b)_{zgz^{-1}}. \]
	Thus $(a \wedge b) \in L_2$. Therefore
	\begin{equation*}
	\begin{split}
		N(\eta)^{a_z}(g) &= a \wedge N(\eta)(zgz^{-1}) \\
		&= a \wedge (\vee \{ b \mid b \in L_2 \}) \\
		&\geq a \wedge \{\vee \{ a \wedge b \mid b \in L_1 \}\} \\
		&= a \wedge \{\vee \{ c \mid c \in L_1\}\} \\
		&= a \wedge N(\eta^{a_z})(g).
	\end{split}
	\end{equation*}
    We conclude that 
    \[ N(\eta)^{a_z}(g) = a \wedge N(\eta^{a_z})(g). \]
\end{proof}

\noindent The following result follows immediately from the above theorem:

\begin{corollary}
	Let $H$ be a subgroup of a group $G$ and $x \in G$. Then,
	\[ N({(1_H)}^{1_x}) = (N(1_H))^{1_x}. \]
\end{corollary}

\begin{example} \label{ex_conjnorm}
	Let $G = D_{16}$, where $D_{16}$ denotes the dihedral group of order $16$, that is,
	\[ D_{16} = \langle r,s \mid r^8 = s^2 = e, rs = sr^{-1} \rangle. \]
	The dihedral group $D_8$ of order $8$ be the dihedral subgroup of $D_{16}$,  where
	\[ D_8 = \langle r^2, s \mid (r^2)^4 = s^2 = e \rangle. \]
	Define $\mu : D_{16} \rightarrow [0,1]$ as follows: for all $z \in D_{16}$,
	\[ \mu(z) = \begin{cases}
		\frac{1}{2}	&\text{if } z \in D_8,\\
		\frac{1}{8} &\text{if } z \in D_{16} \setminus D_8.
	\end{cases} \]
	Note that every non-empty level subset $\mu_t$ is a subgroup of $D_{16}$. Hence by Theorem \ref{lev_gp}, $\mu \in L(D_{16})$. Now, we define $\eta \in L^{\mu}$ as follows:
	\[ \eta(z) = \begin{cases}
		\frac{1}{4} &\text{if } z \in \langle s \rangle, \\
		\frac{1}{16} &\text{if } z \in D_8 \setminus \langle s \rangle, \\
		\frac{1}{32} &\text{if } z \in D_{16} \setminus D_8.
	\end{cases} \]
	Since every non-empty level subset $\eta_t$ is a subgroup of $\mu_t$, by Theorem \ref{lev_sgp}, $\eta \in L(\mu)$. Now, let $a = 1/12$. Then, $a_r \in \mu$, since $\mu(r) = 1/8 \geq 1/12 = a$. We evaluate $N(\eta^{a_r})$ and $N(\eta)^{a_r}$ and show that $N(\eta)^{a_r}(z) = a \wedge N(\eta^{a_r})(z)$ for all $z \in D_{16}$. 
	
	\noindent Firstly, if $H$ denotes the subgroup $\{e, r^4, s, sr^4 \}$ of $D_{16}$, then it is easy to see that 
	\[ N(\eta)(z) = \begin{cases}
		\frac{1}{2} &\text{if } z \in H, \\
		\frac{1}{16} &\text{if } z \in D_{16} \setminus H.
	\end{cases} \]
	Thus the conjugate $L$-subgroup $N(\eta)^{a_r}$ of $N(\eta)$, defined as
	\[ N(\eta)^{a_r}(z) = a \wedge N(\eta)(rzr^{-1}), \]
	is given by
	\[ N(\eta)^{a_r}(z) = \begin{cases}
		\frac{1}{12} &\text{if } z \in H, \\
		\frac{1}{16} &\text{if } z \in D_{16} \setminus H.
	\end{cases} \] 
	Now, the conjugate $L$-subgroup $\eta^{a_r}$ of $\eta$ is given by
	\[ \eta^{a_r}(z) = \begin{cases}
		\frac{1}{12} &\text{if } z \in \langle sr^6 \rangle, \\
		\frac{1}{16} &\text{if } z \in D_8 \setminus \langle sr^2 \rangle, \\
		\frac{1}{32} &\text{if } z \in D_{16} \setminus D_8.
	\end{cases} \]
	Thus the normalizer of $N(\eta^{a_r})$ is 
	\[ N(\eta^{a_r})(z) = \begin{cases}
		\frac{1}{2} &\text{if } z \in H, \\
		\frac{1}{16} &\text{if } z \in D_{16} \setminus H.
	\end{cases} \]
	From this, we can easily see that $N(\eta)^{a_r}(z) = a \wedge N(\eta^{a_r})(z)$ for all $z \in D_{16}$.
\end{example}

\begin{proposition}\label{inc_lvl}
	Let $\eta, \nu \in L(\mu)$. Then, $\eta \subseteq \nu$ if and only if $\eta_t \subseteq \nu_t$ for all $t \leq \text{tip }(\eta)$.
\end{proposition}

\begin{lemma}\label{lemma_con}
	Let $\eta$ be an $L$-subgroup of $\mu$ and $a_z \in \mu$. Then, $\eta^{a_z} \subseteq \eta$ if and only if $\eta^{a_{z^{-1}}} \subseteq \eta$. 
\end{lemma}
\begin{proof}
	 Suppose that $\eta^{a_z} \subseteq \eta$. Then, by Proposition \ref{inc_lvl}, $(\eta^{a_z})_t \subseteq \eta_t$ for all $t \leq \text{tip }(\eta^{a_z}) = a \wedge \eta(e)$. By Theorem \ref{lvl_conj}, $(\eta^{a_z})_t = {\eta_t}^{z^{-1}}$ for all $t \leq a \wedge \eta(e)$. Thus
	\[ {\eta_t}^{z^{-1}} \subseteq \eta_t \text{ for all } t \leq a \wedge \eta(e). \]
	This implies $z^{-1} \in N(\eta_t)$ for all $t \leq a \wedge \eta(e)$. Since the normalizer of $\eta_t$ is a subgroup of $G$, $z \in N(\eta_t)$ for all $t \leq a \wedge \eta(e)$. Thus
	\[ {\eta_t}^z \subseteq \eta_t \text{ for all } t \leq a \wedge \eta(e). \]
	Again, by Theorem \ref{lvl_conj}, $(\eta^{a_{z^{-1}}})_t = {\eta_t}^z$ for all $t \leq a \wedge \eta(e) $. Thus $(\eta^{a_{z^{-1}}})_t \subseteq \eta_t$ for all $t \leq a \wedge \eta(e)= \text{tip }(\eta^{a_z})$. Hence by Proposition \ref{inc_lvl}, $\eta^{a_{z^{-1}}} \subseteq \eta$.
	
	\noindent  The converse part can be obtained by replacing $z$ by $z^{-1}$ in the above exposition.	
\end{proof}

\noindent We wish to characterize the concept of normalizer of an $L$-subgroup with the help of the notion of conjugacy like their classical counterparts. In order to achieve this goal, we prove the following lemma:
\begin{lemma}
	Let $\eta$ be an $L$-subgroup of $\mu$ and $a_z \in \mu$. Then,
	\[
	\eta \circ a_z = a_z \circ \eta \text{~if and only if~} \eta^{a_z} \subseteq \eta.
	\]
\end{lemma}
\begin{proof}
	Let $\eta \circ a_z = a_z \circ \eta$ and $x\in G$. Then, consider
	\begin{eqnarray*}
		\eta^{a_z}(x) &=& a \wedge \eta (zxz^{-1})\\
		&=& (\eta \circ a_z)(zx)\\
		&=& (a_z \circ \eta)(zx)~\qquad\qquad\qquad \text{(by the hypothesis)}\\
		&=& a \wedge \eta (z^{-1}(zx))\\
		&=&a \wedge \eta(x)\\
		&\leq& \eta(x).
	\end{eqnarray*}
	Thus $\eta^{a_z} \subseteq \eta$. \\
	 Let $\eta^{a_z} \subseteq \eta$ . We shall show that . $\eta \circ {a_z} = a_z \circ \eta$. Let $x\in G$ and consider
	\begin{eqnarray*}
		(\eta \circ {a_z})(x) &=& a \wedge \eta (xz^{-1})\\
		&=& a \wedge \eta (z (z^{-1}x)z^{-1})\\
		&=& a \wedge \eta^{a_z}(z^{-1}x)\\
		&\leq& a \wedge \eta(z^{-1}x) \qquad\qquad ~\text{(by the hypothesis,~$\eta^{a_z} \subseteq \eta$)}\\
		&=& (a_z \circ \eta)(x).
	\end{eqnarray*}
	Thus $\eta \circ a_z \subseteq   a_z\circ \eta$. For the reverse inclusion, note that by Lemma \ref{lemma_con}, $\eta^{a_{z^{-1}}} \subseteq \eta$. Thus
	\begin{eqnarray*}
		(a_z \circ \eta)(x) &=& a \wedge \eta (z^{-1}x)\\
		&=& a \wedge \eta (z^{-1} (xz^{-1})z)\\
		&=& a \wedge \eta^{a_{z^{-1}}}(xz^{-1})\\
		&\leq& a \wedge \eta(xz^{-1}) \qquad\qquad ~\text{(since $\eta^{a_{z^{-1}}} \subseteq \eta$)}\\
		&=& (\eta \circ a_z)(x).
	\end{eqnarray*}
	Thus $a_z \circ \eta \subseteq \eta \circ a_z$ and we conclude that $a_z \circ \eta = \eta \circ a_z$.
\end{proof}

\noindent Thus in view of the above lemma, we have the following definition of the normalizer $N(\eta)$ of an $L$-subgroup $\eta$ of $\mu$:

\begin{definition}\label{def_conjnorm}
	Let $\eta \in L(\mu)$. The normalizer of $\eta$ in $\mu$, denoted by $N(\eta$), is the $L$-subgroup defined as follows:
	\[ N(\eta) = \bigcup \left\{ a_z \in \mu \mid  \eta^{a_z} \subseteq \eta \right\}. \]
\end{definition}

\noindent We demonstrate the above definition of the normalizer with the following example:

\begin{example}
	Let $G=D_{16}$ and let $\mu$ and $\eta$ be the $L$-subgroups of $G$ defined in Example \ref{ex_conjnorm}. We use Definition \ref{def_conjnorm} to determine the normalizer of $\eta$ in $\mu$. Note that for $x \in G$,
	\begin{equation}\label{ex_defnorm}
		 \bigcup \left\{ a_z \in \mu \mid  \eta^{a_z} \subseteq \eta \right\}(x) = \bigvee \left\{ a \in L \mid a \leq \mu(x) \text{ and } \eta^{a_x} \subseteq \eta \right\} . 
	\end{equation}
	Taking $x = s$ in Equation \ref{ex_defnorm}, we see that for all $g \in G$,
	\[ \eta^{a_s}(g) = a \wedge \eta(sgs^{-1}) = a \wedge \eta(g). \]
	Thus $\eta^{a_s}(g) \leq \eta(g)$ for all $a \in L$. This implies
	\[ \bigvee \left\{ a \in L \mid a \leq \mu(s) \text{ and } \eta^{a_s} \subseteq \eta \right\} = \mu(s) = \frac{1}{2}. \]
	Next, taking $x=r$ in Equation \ref{ex_defnorm},
	\begin{equation*}
	\begin{split}
		\eta^{a_r}(g) &= a \wedge \eta(rgr^{-1}) \\
		& = \begin{cases}
			a \wedge \frac{1}{4} &\text{if } x \in \langle sr^6 \rangle, \\
			a \wedge \frac{1}{16} &\text{if } x \in D_8 \setminus \langle sr^6 \rangle,\\
			a \wedge \frac{1}{32} &\text{if } x \in D_{16} \setminus D_8.
		\end{cases}
	\end{split} 	 
	\end{equation*}
	Hence it is clear that $\eta^{a_r} \subseteq \eta$ if and only if $a \leq \frac{1}{16}$. Thus
	\[ \bigvee \left\{ a \in L \mid a \leq \mu(r) \text{ and } \eta^{a_r} \subseteq \eta \right\} =  \frac{1}{16}. \]
	Using similar computations, we get
	\[ \bigcup \left\{ a_z \in \mu \mid  \eta^{a_z} \subseteq \eta \right\}(x) = \begin{cases}
		\frac{1}{2} &\text{if } x \in H, \\
		\frac{1}{16} &\text{if } x \in D_{16} \setminus H.
	\end{cases} \]
	which is the normalizer N($\eta$) of $\eta$ in $\mu$ obtained in Example \ref{ex_conjnorm}.
\end{example}

\section{Conclusion}
In classical group theory, the conjugate subgroups play an indispensable role in the studies of normality. In fact, the normalizer of a subgroup $H$ of a group $G$ can be defined as the collection of all elements $x$ in $G$ such that the conjugate $H^x$ of $H$ with respect to $x$ is contained in $H$. While the notion of the normalizer of an $L$-subgroup was efficiently introduced in \cite{ajmal_nor}, the concept of the conjugate of an $L$-subgroup that was compatible with the normality and the normalizer of $L$-subgroups was absent. We have, in this paper, succeeded in providing such a notion of the conjugate. Hopefully, the concepts introduced in this paper will play a significant role in further studies of $L$-algebraic structures. 

The research in the discipline of fuzzy group theory came to a halt after Tom Head’s metatheorem and subdirect product theorem. This is because most of the concepts and results in the studies of fuzzy algebra could be established through simple applications of the metatheorem and the subdirect
product theorem. However, the metatheorem and the subdirect product theorems are not applicable in the $L$-setting. Hence we suggest the researchers pursuing studies in these areas to investigate the properties of $L$-subalgebras of an $L$-algebra rather than $L$-subalgebras of classical algebra.

\section*{Acknowledgements}
The second author of this paper was supported by the Senior Research Fellowship jointly funded by CSIR and UGC, India during the course of development of this paper.

\end{document}